\documentclass{amsart}
\usepackage{amsmath, amsthm, amssymb}
\usepackage[all]{xypic}
\usepackage{hyperref}

\newcommand{\mM}{\underline{M}}
\newcommand{\mN}{\underline{N}}
\newcommand{\mZ}{\underline{\mathbb Z}}

\newcommand{\mA}{\underline{A}}
\newcommand{\mI}{\underline{I}}

\newcommand{\mE}{\underline{E}}
\newcommand{\Z}{\mathbb Z}
\newcommand{\N}{\mathbb N}
\newcommand{\MU}{MU}
\newcommand{\EP}{\tilde{E}\mathcal P}
\newcommand{\cP}{\mathcal P}
\newcommand{\R}{\mathbb R}
\DeclareMathOperator{\sdim}{scon}
\newcommand{\res}{Res}
\newcommand{\boxprod}{\,{\scriptstyle \boxtimes}\,}

\newtheorem{theorem}{Theorem}[section]
\newtheorem{thm}[theorem]{Theorem}
\newtheorem{cor}[theorem]{Corollary}
\newtheorem{defn}[theorem]{Definition}
\newtheorem{prop}[theorem]{Proposition}
\newtheorem{conj}[theorem]{Conjecture}

\newtheorem{warning}[theorem]{Warning}
\newtheorem{example}[theorem]{Example}
\newtheorem{remark}[theorem]{Remark}
\newtheorem{lemma}[theorem]{Lemma}

\newcommand{\EM}{Eilenberg-Mac Lane}

\begin{document}

\title{The equivariant slice filtration: a primer}

\author{Michael A.~Hill}
\email{mikehill@virginia.edu}
\address{Department of Mathematics,
University of Virginia,
Charlottesville, VA 22904}
\thanks{The author was supported by NSF DMS--0906285, DARPA FA9550-07-1-0555, and the Sloan Foundation}



\begin{abstract}
We present an introduction to the equivariant slice filtration. After reviewing the definitions and basic properties, we determine the slice-connectivity of various families of naturally arising spectra. This leads to an analysis of pullbacks of slices defined on quotient groups, producing new collections of slices. Building on this, we determine the slice tower for the Eilenberg-Mac Lane spectrum associated to a Mackey functor for a cyclic $p$-group. We then relate the Postnikov tower to the slice tower for various spectra. Finally, we pose a few conjectures about the nature of slices and pullbacks.
\end{abstract}

\maketitle

\section{Introduction}
\subsection{Background}
Essential to the solution with Hopkins and Ravenel to the Kervaire Invariant One Problem is the construction of a new natural filtration of an equivariant spectrum, the slice filtration \cite{HHR:Kervaire}. This is a generalization of Dugger's slice filtration for $C_{2}$-equivariant homotopy theory, and it is analogous to the motivic slice filtration of Voevodsky \cite{Dugger}, \cite{Voe:ZeroSlice}, \cite{HoMor}. The basic idea is simple: mirror the Postnikov tower construction using a different collection of representation spheres (giving a different notion of ``connected'' and ``co-connected''). The slice and Postnikov towers are obviously closely related, but the exact relationship is unclear.

This filtration is especially nice on the spectra $\MU^{(n)}$ used in the proof and on the relevant localizations thereof. For these, the homotopy groups of the layers compute the homology and cohomology of regular representation spheres. This allowed us to determine the vanishing of $\pi_{-2}$ of the relevant localizations. In some sense, we were very lucky; the slice filtration is very mysterious for many spectra other than these localizations of $\MU^{(n)}$.

The goal of this paper is two-fold:
\begin{enumerate}
\item serve as a travelogue for the slice filtration and slice spectral sequence, recording basic properties and interpretations, and
\item provide a bestiary of slices and the slice tower for a general Mackey functor for cyclic $p$-groups.
\end{enumerate}
We begin in \S\ref{sec:Basics} by recalling the basic facts about the slice filtration, drawing heavily from \cite{HHR:Kervaire}. Starting in \S\ref{sec:Dimensions}, we set a course for uncharted territory. In \S\ref{sec:Dimensions},  we will address when there are natural connections between the dimension of a $CW$-complex or of a representation sphere and its slice-connectivity. We will relate slices to pullbacks along surjective maps of groups in \S\ref{sec:Pullback}. In \S\ref{sec:Burnside}, we take an algebraic digression, using this to directly compute the slice tower of the {\EM} spectrum for the Burnside Mackey functor $\mA$ and for a general Mackey functor $\mM$ for a cyclic group of prime-power order. Finally in \S\ref{sec:Geometric}, we build on the computations for $H\mA$, introducing an algebraic criterion on the homotopy groups of a $G$-spectrum $X$ which allows for complete determination of slices of $X$. This provides a class of spectra for which the connection between the slice and Postnikov towers is completely understood. We use this to provide an interpretation of the slice tower for a general spectrum.

All of our discussion will take place in the genuine equivariant stable homotopy category for a finite group $G$. In particular, we have natural transfers associated to all subgroups. Good references are the work of May et al \cite{LMayS}, \cite{MandellMay}, \cite{May:Alaska}, unpublished notes of Schwede \cite{Schwede}, and the appendix to Hill-Hopkins-Ravenel \cite{HHR:Kervaire}. We use freely and liberally the language of Mackey functors. In particular, all homotopy groups are Mackey functor valued, and we will follow standard notation, writing these with an underline. We found papers and notes of Lewis especially helpful \cite{Lewis:Projective}, \cite{Lewis:Green}.

\subsection{Notation}
Throughout all that follows $G$ will be a finite group. When it appears, $N$ will always denote a normal subgroup, while $H$ and $K$ will be used for subgroups that are not necessarily normal.

The real regular representation of $G$ will be denoted $\rho_{G}$, and if there is no ambiguity, we shall drop the subscript. If $X$ is a finite $G$ set, then we denote the permutation representation generated by $X$ by $\rho_{X}$. The quotient of $\rho$ or $\rho_{G}$ by the trivial subrepresentation will be denoted $\bar{\rho}_{G}$.

\subsection*{Acknowledgements}
The author is indebted to Andrew Blumberg, Tyler Lawson, and Doug Ravenel for many helpful conversations and comments on earlier drafts.

\section{Basic properties of the slice filtration}\label{sec:Basics}
In this section, we recall the definition and basic properties of the slice filtration. Essentially all of this material is from \cite{HHR:Kervaire}, especially sections 3 and 4.

\begin{defn}
For each integer $n$, let $\tau_{\geq n}$ denote the localizing subcategory of $G$-spectra generated by $G_+\wedge_H S^{k\rho_H-\epsilon}$, where $H$ ranges over all subgroups of $G$, where $k\cdot |H|-\epsilon\geq n$, and where $\epsilon=0,1$.

If $X$ is an object of $\tau_{\geq n}$, we say ``$X \geq n$'' and we say the ``slice-connectivity of $X$'', written $\sdim(X)$ is greater than or equal to $n$.

If $X$ is in $\tau_{\geq n}$ but not in $\tau_{\geq n+1}$, we will say that the slice-connectivity of $X$ is $n$.
\end{defn}

\noindent If we are considering $\tau_{\geq n}$ for various groups, we shall adorn this symbol with a superscript distinguishing between them: $\tau_{\geq n}^G$.

What does ``localizing subcategory'' mean? Simply put, this is the category of acyclics for a localization functor on $G$-spectra \cite{DrorFarjoun}. In particular, we have the following properties.
\begin{enumerate}
\item The category $\tau_{\geq n}$ is a full subcategory, and if $X$ is weakly equivalent to an object in $\tau_{\geq n}$, then $X$ is an object in $\tau_{\geq n}$.
\item $\tau_{\geq n}$ is closed under cofibers: If $X$ and $Y$ are objects in $\tau_{\geq n}$ and $f\colon X\to Y$ is a map, then the cofiber $C(f)$ is an object in $\tau_{\geq n}$.
\item $\tau_{\geq n}$ is closed under extensions: If $X$ and $Z$ are objects in $\tau_{\geq n}$, and if $X\to Y\to Z$ is a cofiber sequence, then $Y$ is an object in $\tau_{\geq n}$.
\item $\tau_{\geq n}$ is closed under retracts and infinite wedges.
\item $\tau_{\geq n}$ is closed under directed colimits.
\end{enumerate}

\begin{warning}
It is {\emph{very}} important that these are not triangulated subcategories. They are closed under suspension and cofibers but not desuspension and fibers. In fact, it is not even the case that desuspension always takes $\tau_{\geq n}$ to $\tau_{\geq (n-1)}$. Remark~\ref{rmk:RepDimen} below illustrates this.
\end{warning}

Associated to a localizing subcategory is a localization / nullification functor for which the localizing category is the category of acyclics \cite{DrorFarjoun}.

\begin{defn}
Let $P^{n-1}(-)$ be the localization functor associated to $\tau_{\geq n}$.
\end{defn}

\begin{defn}
We say ``$X$ is less than or equal to $(n-1)$'' if the localization map $X\to P^{n-1}(X)$ is an equivalence.
\end{defn}

There is an obvious inclusion of full subcategories $\tau_{\geq (n+1)}\subset\tau_{\geq n}$. This gives us natural transformations $P^n(-)\to P^{n-1}(-)$, and so to any spectrum $X$, we have a naturally associated tower.

\begin{defn}
The ``slice tower'' of $X$ is the tower with stages $P^n(X)$ and maps the natural maps $P^n(X)\to P^{n-1}(X)$. The fiber of the map $P^n(X)\to P^{n-1}(X)$ is the ``$n$\textsuperscript{th} slice of $X$'', denoted $P_n^n(X)$.
\end{defn}

In forming the nullification tower, we implicitly killed the $G$-space of all maps from objects in $\tau_{\geq n}$ to $X$ in order to form $P^{n-1}(X)$. However, we can see equivariance much more easily as follows.

\begin{prop}
Let $i_H^*$ denote the forgetful functor from $G$-spectra to $H$-spectra, and let $G_+\wedge_H(-)$ denote its left adjoint, induction. Then we have natural inclusions of full sub-categories
\[
i_H^*\tau_{\geq n}^G\subset \tau_{\geq n}^H\quad\text{ and }\quad G_+\wedge_H(\tau_{\geq n}^H)\subset\tau_{\geq n}^G.
\]

The localizing subcategory generated by $i_H^*(\tau_{\geq n}^G)$ is $\tau_{\geq n}^H$.
\end{prop}

This is actually immediate from the definitions. The restriction of any of the generators of $\tau_{\geq n}^G$ to $H$ is a wedge of generators of $\tau_{\geq n}^H$. Similarly, if we induce up a generator of $\tau_{\geq n}^H$, then we get a generator of $\tau_{\geq n}^G$. This proves the two inclusions. For the final part, $X$ is a wedge summand of $i_H^*(G_+\wedge_H X)$, and therefore all generators of $\tau_{\geq n}^H$ are in the localizing subcategory generated by $i_{H}^{\ast}(\tau_{\geq n}^{G})$.

\begin{cor}
The restriction to $H$ of the slice tower of $X$ is the slice tower of the restriction of $X$ to $H$:
\[
P^ni_H^*(X)=i_H^*P^n(X).
\]
\end{cor}

For positive $n$, the slice tower receives a map from the Postnikov tower of the same dimension. 
\begin{prop}\label{prop:Spheres}
For all subgroups $H$ of $G$ and for all $n\geq -1$, the induced sphere $G_+\wedge_HS^n$ is in $\tau_{\geq n}$.
\end{prop}
This is not difficult to show using induction on $|G|$ and  the closure of $\tau_{\geq n}$ under extensions, since $G_+\wedge_HS^n$ is the bottom cell in $G_+\wedge_HS^{n\rho_H}$. We will use this technique exclusively. This result shows that the slice tower refines the Postnikov tower in non-negative degrees.

\begin{cor}
For all $X$ and for all $n\geq -1$, $P^{n-1}(X)$ is $n$-coconnected.
\end{cor}
The case for $n$ negative is also important both in determining the colimit of the slice tower and in our later analysis of special slices. We postpone our discussion of these cases briefly.

Now here are the only examples in which we completely understand $\tau_{\geq n}$.

\begin{example}
The generators of $\tau_{\geq 0}$ which are not zero-connected are $G_+\wedge_H S^{0\rho_H}$ and $G_+\wedge_HS^{\bar{\rho}_{H}}$. Since $G_+\wedge_H S^0=G/H_+$, and since all spectra which are $(-1)$-connected are weakly equivalent to $G$-CW-spectra, we learn that
\[
\tau_{\geq 0}=\{\text{(-1)-connected $G$-spectra.}\}
\]
\end{example}

\begin{example}
Just as with $\tau_{\geq 0}$, we can understand $\tau_{\geq -1}$. Here the generators are all in the localizing subcategory generated by $\Sigma^{-1} G/H_+$, and so
\[
\tau_{\geq -1}=\{\text{(-2)-connected $G$-spectra.}\}
\]
\end{example}

Together this gives an important corollary:

\begin{cor}
The $(-1)$-slice of any spectrum $X$ is the $(-1)$-Postnikov layer: 
\[
P^{-1}_{-1}(X)\simeq\Sigma^{-1}H\underline{\pi_{-1}}(X).
\]
\end{cor}

\begin{cor}
For any Mackey functor $\mM$, $\Sigma^{-1}H\mM$ is a $(-1)$-slice.
\end{cor}

We also have an algebraic description of $0$-slices. A complete proof is in \cite[Lemma 3.2]{HHR:Kervaire}; we sketch a proof here.

\begin{thm}\label{thm:zeroslices}
The category of $0$-slices is the category of Mackey functors $\mM$ such that all restrictions are injections.
\end{thm}
\begin{proof}[Sketch of Proof]
Since $G_+\wedge_HS^1$ is in $\tau_{\geq 1}$ and since $\tau_{\geq 0}$ is the category of $(-1)$-connected spectra, we know that the zero slices are all of the form $H\mM$ for some Mackey functor $\mM$. We need only determine which are allowed.

The essential step is the equality
\[
[S^{\rho_G-1},H\mM]_G=\{x\in\mM(G/G) | \res_H^G(x)=0\quad\forall H\subsetneq G\},
\]
i.e. the collection of elements in $\mM(G/G)$ that map to zero under all restriction maps. Since our localizing subcategory is closed under induction and restriction, this shows that any element of $\mM(G/H)$ which restricts to zero in $\mM(G/K)$ for all subgroups $K\subsetneq H$ lifts to an element of 
\[
[G_{+}\wedge_{H}S^{\bar{\rho}_{H}},H\mM].
\] 
Thus if these elements are all zero, so are these collections of maps. It is easy to see that this implies that all restriction maps are injections.
\end{proof}

\begin{remark}
It is not the case that any $x\in\mM(G/G)=\underline{\pi_{0}}H\mM$ which restricts to zero in some $\mM(G/H)$ extends to a map $S^{\bar{\rho}_{G}}\to H\mM$. However, we do know that by iteratively killing these elements off, in some eventual cofiber, the element $x$ will restrict to zero in all proper subgroups.
\end{remark}

The previous theorem tells us the $0$-slice of any spectrum $X$.

\begin{cor}
The $0$-slice of a $G$-spectrum $X$ is $HP_0^0\underline{\pi_0}(X)$, where for a Mackey functor $\mM$, $P_0^0\mM$ is the largest quotient of $\mM$ in which all restriction maps are monomorphisms.

The value on $G/H$ of this quotient is $Im\big(\mM(G/H)\to\mM(G/\{e\})\big)$.
\end{cor}

\begin{cor}
A $G$-spectrum $X$ is in $\tau_{\geq 1}$ if and only if $X$ is in $\tau_{\geq 0}$ and 
\[
\underline{\pi_{0}(X)}(G/\{e\})=0.
\]
\end{cor}

After this point, we no longer have nice, easy Mackey functor descriptions. However, many of the categories of slices are secretly naturally equivalent to the aforedescribed categories of Mackey functors. To see this, we introduce the final important tool we will use: slices commute with suspension by copies of the regular representation in the same way that the Postnikov sections commute with ordinary suspension.

\begin{thm}\label{thm:suspensions}
For any $X$, for any $n$, and for any $k$, 
\[
P^{n+k|G|}(\Sigma^{k\rho_G} X)=\Sigma^{k\rho_G}P^n(X),
\]
and hence
\[
P^{n+k|G|}_{n+k|G|}(\Sigma^{k\rho_G}X)=\Sigma^{k\rho_G}P^n_n(X).
\]
\end{thm}
This is simply because we have an underlying statement about the localizing subcategories.

\begin{lemma}
For all $n$ and for all $k$, we have
\[
S^{k\rho_G}\wedge\tau_{\geq n}=\tau_{\geq (n+k|G|)}.
\]
\end{lemma}
This result follows by noticing that smashing with $S^{n\rho_G}$ induces a bijection on isomorphism classes of generators and commutes with all of the properties.

\begin{cor}
The category of $(k|G|-1)$-slices is equivalent to the category of Mackey functors via the suspension by $k\rho_G$.
\end{cor}

\begin{cor}
The category of $(k|G|)$-slices is equivalent to the category of Mackey functors in which all restrictions are injections via the suspension by $k\rho_G$.
\end{cor}

Suspension invariance of the categories $\tau_{\geq n}$ will allow us to determine the slice-connectivity of negative spheres.
\begin{prop}\label{prop:NegativeSpheres}
For $n\geq 1$, the spectrum $G_{+}\wedge_{H}S^{-n}$ is in $\tau_{\geq -(n-1)|H|-1}$ but not in $\tau_{\geq -(n-1)|H|}$.
\end{prop}
\begin{proof}
Since the spectrum is induced and since induction preserves slice-connectivity, it suffices to prove this for $H=G$. By our suspension invariance, showing that $S^{-n}$ is in $\tau_{\geq -(n-1)|G|-1}$ but not in $\tau_{\geq -(n-1)|G|}$ is equivalent to showing that  $S^{(n-1)\rho-n}$ is in $\tau_{\geq -1}$ but not in $\tau_{\geq 0}$. This virtual representation sphere is $S^{(n-1)\bar{\rho}-1}$, and by looking at fixed points, we see that this is $(-2)$-connected but not $(-1)$-connected.
\end{proof}

With our understanding of connectivity and of which spheres occur in which slice-connectivities, convergence of the slice tower is actually relatively straightforward to prove. All of the generators of $\tau_{\geq n}$ are at least $(n/|G|-1)$-connected if $n\geq 0$ and $(n-1)$-connected if $n\leq -1$. This means that all elements in $\tau_{\geq n}$ share these connectivity lower bounds. Thus $\bigcap \tau_{\geq n}$ is the full subcategory of weakly contractible spectra. Proposition~\ref{prop:NegativeSpheres} shows that $\bigcup\tau_{\geq n}$ is the full subcategory of bounded below spectra. Thus the limit of the slice tower is $X$ and the colimit is contractible. 

We close the section with a few remarks about slice filtration and smash products. In general, if $X\in\tau_{\geq n}$ and $Y\in\tau_{\geq m}$, then we know very little about the slice-connectivity of $X\wedge Y$ (even if both $m$ and $n$ are non-negative). There are several cases in which we know more.

\begin{prop}\label{prop:SmashPositive}
If $X$ is in $\tau_{\geq 0}$ and $Y$ is in $\tau_{\geq n}$, then $X\wedge Y$ is in $\tau_{\geq n}$.
\end{prop}
\begin{proof}
The category $\tau_{\geq 0}$ is generated by $G/H_{+}$ for $H$ a subgroup. Smashing $Y$ with this is equivalent to taking the restriction to $H$ of $Y$ and then inducing back to $G$, and we have already seen that $\tau_{\geq n}$ is closed under these operations. 
\end{proof}

\begin{cor}\label{cor:SmashOrderG}
If $X$ is in $\tau_{\geq k|G|}$ and $Y$ is in $\tau_{\geq n}$, then $X\wedge Y$ is in $\tau_{\geq k|G|+n}$.
\end{cor}
\begin{proof}
By suspension invariance, $\Sigma^{-k\rho_{G}}X\in\tau_{\geq 0}$. Proposition~\ref{prop:SmashPositive} then tells us that $\Sigma^{-k\rho_{G}}X\wedge Y\in\tau_{\geq n}$, and by suspension invariance again, $X\wedge Y\in\tau_{k|G|+n}$.
\end{proof}

\section{Slice-connectivity and underlying dimension}\label{sec:Dimensions}
At this point, we leave the more familiar waters of Hill-Hopkins-Ravenel and begin our more specific discussions of examples of slices. We start with a surprisingly useful generalization of Proposition~\ref{prop:Spheres}.

\begin{lemma}\label{lem:Suspensions}
If $Y$ is in $\tau_{\geq m}$, then $\Sigma Y$ is in $\tau_{\geq (m+1)}$.
\end{lemma}
\begin{proof}
The proof of Proposition~\ref{prop:Spheres} actually shows this. We quickly review it here. The proof is by induction on the order of the group. The base-case is the non-equivariant statement that if $Y$ is $(m-1)$-connected, then $\Sigma Y$ is $m$-connected. We have a cofiber sequence
\[
\Sigma S(\rho-1)_{+}\to S^{1}\to S^{\rho},
\]
so smashing with $Y$ yields another cofiber sequence. The group $G$ acts without fixed points on $\Sigma S(\rho-1)_{+}$, and it is $0$-connected. Thus it is built out of cells of the form $G/H_+\wedge S^k$ for $k\geq 1$ and $H\subsetneq G$. By the induction hypothesis, 
\[
\Sigma S(\rho-1)_{+}\wedge Y
\]
is in $\tau_{\geq (m+1)}$. Similarly $S^{\rho}\wedge Y$, being in $\tau_{\geq (|G|+m)}$, is in $\tau_{\geq (m+1)}$. 
\end{proof}

This gives two closer connections between the slice and Postnikov towers. First, we can relate connectivity and slice-connectivity.

\begin{cor}
If $X$ is $(n-1)$-connected with $n\geq 0$, and if $Y$ is in $\tau_{\geq m}$, then 
$X\wedge Y$ is in $\tau_{\geq (n+m)}$.
\end{cor}
\begin{proof}
The connectivity assumption shows that $X$ can be built out of suspensions of $S^{n}$ and induced cells. Iterated applications of Lemma~\ref{lem:Suspensions} then gives the result.
\end{proof}

We can also slightly refine Theorem~\ref{thm:suspensions}, providing a connection between slices and suspensions. This follows from the inclusion 
\[
S^{1}\wedge\tau_{\geq k}\subset \tau_{\geq (k+1)}.
\]

\begin{cor}
We have a natural map
\[
\Sigma P^{k}(\Sigma^{-1}X)\to P^{k+1}(X).
\]
\end{cor}
\noindent For $X=S^0$ and $k=-1$, this gives the natural map $H\mA\to H\mZ$.

We now develop two closely related criteria which establish some bounds on the slice size of skeleta of a finite $G$ $CW$ spectrum $X$ smashed with an arbitrary spectrum $Y$. The two variants depend on which factor we have tight control of. We first assume some control on $Y$, deducing results about the slice-connectivity of skeleta of $X$.

\subsection{Slice-connectivity of skeleta}
As a bit of notation, if $X$ is an equivariant CW-spectrum, let $X^{[k]}$ denote its $k$-skeleton. We begin by generalizing the result that $S^n$ is in $\tau_{\geq n}$ for positive $n$.

\begin{thm}\label{thm:DimCW}
Let $n$ be a non-negative integer, and let $Y$ be in $\tau_{\geq m}$. If $X$ is an $n$-dimensional $G$ CW-complex such that $X\wedge Y$ is in $\tau_{\geq (n+m)}$, then $X^{[k]}\wedge Y$ is in $\tau_{\geq (k+m)}$ for all $0\leq k\leq n$.
\end{thm}
\begin{proof}
We proceed by induction on the codimension. The base case of codimension $0$ is by definition. Assume that the codimension $(n-k-1)$-skeleton is greater than or equal to $(k+1+m)$. We have a cofiber sequence
\[
\bigvee_{\alpha} G_+\wedge_{H_{\alpha}} S^k\wedge Y\to X^{[k]}\wedge Y\to X^{[k+1]}\wedge Y,
\]
where the wedge is taken over all $(k+1)$-cells $e_{\alpha}$ of $X$. By assumption, the rightmost term is greater than $(k+m)$. By Lemma~\ref{lem:Suspensions}, the leftmost term is greater than or equal to $(k+m)$. Since we are looking at a localizing subcategory, we conclude that the middle term is greater than or equal to $(k+m)$, as required.
\end{proof}

Applying this to $Y=S^{0}$ shows us that if $X$ is an $n$-dimensional complex which is in $\tau_{\geq n}$, then the $k$-skeleton is in $\tau_{\geq k}$.

\begin{remark}
It is obviously not true that every $n$-dimensional $G$-CW-complex is in $\tau_{\geq n}$. If we wedge a copy of $S^{0}$ onto a general element of $\tau_{\geq n}$, then we drop the slice-connectivity to $0$. Even if we insist on irreducible complexes, the result is far from true, as the representation sphere example below shows.
\end{remark}

If $X$ is in $\tau_{\geq m}$, then $\Sigma^{n\rho}X$ is in $\tau_{\geq (np+m)}$. Applying this and Theorem~\ref{thm:DimCW} to the spheres $S^\ell$ for $\ell\geq -1$, we deduce the following corollary.
\begin{cor}
For all $\ell \geq -1$ and for all $k\leq n|G|+\ell$, $\big(S^{n\rho+\ell}\big)^{[k]}$ is greater than or equal to $k$.
\end{cor}

\subsection{Slice-connectivity of representation spheres}
In equivariant homotopy theory, there is a tension between $G$ CW-complexes and representation spheres: certain statements are easier to prove for one or the other. In this case, we have a quite nice dual to the previous theorem but for representation spheres.

\begin{thm}\label{thm:DimRep}
Let $Y$ be in $\tau_{\geq m}$, and let $W$ be a representation such that $S^W\wedge Y$ is in $\tau_{\geq (\dim W+m)}$. Let $V$ is a subrepresentation of $W$ such that $V^G=W^G$. Then 
\[
\sdim(S^V\wedge Y)\geq\min\Big(\big\{\sdim(i_H^\ast S^V\wedge Y) | (W/V)^H\neq \{0\}\big\}\cup \big\{\dim W+m\big\}\Big).
\]

If the minimum is achieved by one of the restrictions, then 
\[
\sdim(S^V\wedge Y)=\min\Big(\big\{\sdim(i_H^\ast S^V\wedge Y) | (W/V)^H\neq \{0\}\big\}\Big).
\]
\end{thm}

\begin{proof}
Let $U$ denote the orthogonal complement of $V$ in $W$. Since $V^G=W^G$, we know $U^G=\{0\}$. Pick a $G$-CW decomposition of $S(U)$. Since $U^G=\{0\}$, we know that there are no fixed cells. The representation sphere $S^{U}$ is the unreduced suspension $S^0\ast S(U)$, and so it therefore has a single trivial cell: the zero cell. Our direct sum decomposition gives rise to a homeomorphism $S^V\wedge Y\wedge S^{U}\cong S^W\wedge Y$. 

Our analysis of the skeleta of $S^{U}$ shows that we get $S^W\wedge Y$ from $S^V\wedge Y$ by attaching ``cells'' of the form $G_{+}\wedge_{H}e^{r+1}\wedge S^{V}\wedge Y$ with $H$ a proper subgroup. The attaching map for this is from $G_+\wedge_H S^r\wedge S^V\wedge Y$, and by Lemma~\ref{lem:Suspensions}, the slice-connectivity of this is greater than or equal to that of $i_H^\ast (S^V\wedge Y)$. Thus by downward induction, $S^V\wedge Y$ can be nested between things of slice-connectivity greater than or equal to the minimum of these and the starting value of $\dim W+m$.

Since for all spectra $X$ and for all subgroups $H$ we have an inequality
\[
\sdim(X)\leq \sdim(i_H^\ast X),
\]
if the minimum for $\sdim(S^V\wedge Y)$ is achieved on some subgroup (rather than $\dim W+m$), then we immediately have equality.
\end{proof}

The inclusion of the term $\dim W+m$ is necessary to cover the cases in which the restriction of $Y$ to any proper subgroup of $G$ is contractible. In this case, $S^V\wedge Y=S^W\wedge Y$, and so they necessarily have the same slice-connectivity. Spectra of this form will be studied more in Section~\ref{sec:Geometric}.

If we add in more conditions, then we can get sharper bounds on the slice-connectivity. The following proposition is identical in proof to the previous theorem. In this, we explicitly assume tighter control of the slice-connectivitys of the domains of the attaching maps for the cells.

\begin{prop}
Let $Y$ be in $\tau_{\geq m}$ and let $V\subset W$ be representations such that 
\begin{enumerate}
\item $V^G=W^G$,
\item $S^W\wedge Y$ is in $\tau_{\geq \dim W+m}$, and
\item  for all $H\subsetneq G$ and $r\in \N$ such that $S^{W/V}$ has a cell of the form $G_{+}\wedge_{H} e^{r+1}$, the restriction $i_H^\ast(S^{V+r}\wedge Y)$ is in $\tau_{\geq \dim V+r+m}$, 
\end{enumerate}
then $S^V\wedge Y$ is in $\tau_{\geq \dim V+m}$.
\end{prop}

It seems like these are unduly hard restrictions, requiring us to know almost everything at the outset. In practice, the second condition is purely representation-theoretic and easy to check based on the dimensions of fixed points. Using Lemma~\ref{lem:Suspensions}, we have a very easy to check sufficient condition for Theorem~\ref{thm:DimRep}.

\begin{cor}\label{cor:weakestequality}
If $Y$ is in $\tau_{\geq m}$, if $S^{W}$ is a representation sphere such that $S^{W}\wedge Y$ is in $\tau_{\geq (\dim W+m)}$, and if $V$ is a subrepresentation of $W$ such that $V^G=W^G$ and the restriction to all proper subgroups of $S^{V}\wedge Y$ is in $\tau_{\geq (\dim V+m)}$, then $S^{V}\wedge Y$ is in $\tau_{\geq(\dim V+m)}$.
\end{cor}

\noindent For cyclic $p$-groups, this becomes a single condition to check: that the restriction to the maximal proper subgroup is of the desired slice-connectivity.

Choosing $Y=S^0$ in the previous theorems allows us to determine for various representation spheres the slice-connectivity.

\begin{cor}
For all finite dimensional representations $V$, there is a $K$ such that $S^{V+k}$ is in $\tau_{\geq \dim V+k}$ for all $k\geq K$.
\end{cor}
\begin{proof}
Choose $k_0$ such that
\[
V+k_0\subset m\rho+\bar{\rho}=W_0
\]
and $(V+k_0)^G=W_0^G$. By definition, $S^{W_0}$ is in $\tau_{\geq\dim W_0}$, and by Lemma~\ref{lem:Suspensions}, for all $\ell\geq 0$, $S^{W_0+\ell}$ is in $\tau_{\geq\dim W_0+\ell}$. By induction on $G$, there is a $k_1$ such that for all $H\subsetneq G$,
\[
\sdim(i_H^\ast S^{V+k_0+k_1})\geq\dim V+k_0+k_1.
\]
By Corollary~\ref{cor:weakestequality}, we conclude that we need only take $K=k_0+k_1$.
\end{proof}

There is one beautiful family in which we can get a slick result. This was originally posed by Strickland, and it gives a clean view of a natural family of representations that have the right slice-connectivity.

\begin{thm}\label{thm:PermutationDimension}
If $X$ is a finite $G$-set and $\epsilon$ is $0$ or $1$, then $S^{\rho_{X}-\epsilon}$ is in $\tau_{\geq |X|-\epsilon}$.
\end{thm}
\begin{proof}
This will be proved by induction on the order of $G$. For $G$ the trivial group (and therefore for any representation sphere induced up from the trivial group), this is obvious since spheres are the representation spheres for the trivial group. 

Now assume that for any proper subgroup of $G$ the result is true, and we can apply Theorem~\ref{thm:DimRep}. Let $k=|X/G|$. Since the permutation representation associated to an orbit $G/H$ embeds as a subrepresentation of the regular representation, we have a natural inclusion
\[
\rho_{X}-\epsilon\subseteq k\rho_{G}-\epsilon.
\]
Moreover, this has the property that
\[
(\rho_{X}-\epsilon)^{G}=\R^{k-\epsilon}=(k\rho_{G}-\epsilon)^{G},
\]
and by definition, $S^{k\rho_{G}-\epsilon}$ is in $\tau_{\geq k|G|-\epsilon}$. Thus the first hypothesis of Theorem~\ref{thm:DimRep} is satisfied. To prove the result, we need only show that the restriction of $S^{\rho_{X}-\epsilon}$ to any proper subgroup is in $\tau_{\geq |X|-\epsilon}$. However, $\rho_{X}$ is the permutation representation $\rho_{X}$, and this restricts to the permutation representation
\[
\rho_{Y}=\rho_{\res_{H}^{G}(X)}.
\]
Our desired representation sphere therefore restricts to $S^{\rho_{Y}-\epsilon}$. By our induction hypothesis, this is greater than or equal to its dimension, and we have proved the result.
\end{proof}

\begin{remark}\label{rmk:RepDimen}
Just as with CW complexes, it is not the case that every representation sphere is greater than or equal to its dimension.  For example, $S^{2\rho-2}$ is not in $\tau_{2|G|-2}$. The $\rho$-desuspension of $S^{2\rho-2}$ is $S^{\rho-2}$ which is in $\tau_{\geq -1}$ but not in $\tau_{\geq 0}$ (since it is not $(-1)$-connected). Thus $S^{2\rho-2}$ is in $\tau_{\geq (|G|-1)}$ but not in $\tau_{\geq |G|}$.
\end{remark}

Using similar methods, we can produce the Spanier-Whitehead dual of Theorem~\ref{thm:PermutationDimension}. Rather than present it in full generality, we will give the form which we will later use.

\begin{thm}\label{thm:NegPermutation}
If $N\lhd G$, then the spectrum $S^{\rho_G-\rho_{G/N}}$ is in $\tau_{\geq |N|-1}$.
\end{thm}

\begin{proof}
We have an inclusion of representations
\[
V=\rho_G-\rho_{G/N}\subset \rho_G-1=W,
\]
and the $G$-fixed points agree. Thus we can find a lower-bound on the size of $S^{\rho_G-\rho_{G/N}}$ by apply Theorem~\ref{thm:DimRep}. The representation $W/V$ is $\rho_{G/N}-1$, and those subgroups which occur in the $G$ CW-decomposition all contain $N$. Thus we need to understand the slice-connectivity of 
\[
i_H^\ast (S^{\rho_G-\rho_{G/N}})=S^{[G:H](\rho_H-\rho_{H/N})}
\]
for $N\subset H\subset G$. By downward induction on $[G:N]$,  it is obvious that the base case of $H=N$ has the minimum size, and for this, we are considering
\[
S^{[G:N](\rho_N-1)}.
\]

By Remark~\ref{rmk:RepDimen}, this is in $\tau_{\geq |N|-1}$ but not $\tau_{\geq |N|}$.
\end{proof}

A consequence of Corollary~\ref{cor:PullBackZeroSlices} below is that this bound is sharp for $N\lhd G$. For non-normal subgroups $H$, we do not know the expected bound.

There is a somewhat depressing corollary to this: the norm and slice filtration can behave less ideally than we might have hoped. If $E$ is a connective genuine equivariant commutative ring spectrum, then so is $P^k(E)$ for all $k\geq 0$. In particular, we have norm maps
\[
N_H^G i_H^\ast P^k(E)\to P^k(E)
\]
for all subgroups $H$ and for all $k\geq 0$.

For spectra like those that occur in the solution to the Kervaire invariant one problem, computation evidence suggests that for some $E$, there is a natural lift (in commutative rings)
\[
\xymatrix{
{} & {P^{k[G:H]}E}\ar[d] \\
{N_H^G i_H^\ast P^k E}\ar[ur]^{N}\ar[r]_{N} & {P^kE}}.
\]
For very stupid reasons, this holds for all commutative $E$ for $k=0$, but our result above shows that this cannot hold for a general $E$ and $k$. The norm of $S^{\rho_N-1}$ is $S^{\rho_G-\rho_{G/N}}$, and we saw that this is only in $\tau_{\geq |N|-1}$. So if we had to kill maps from $S^{\rho_N-1}$ into $E$, then the norm of those maps might not be killed at the desired lift.

\section{Pullbacks and slices thereof}\label{sec:Pullback}
In this section, we will describe an interesting and useful family of slices: pullbacks. In the subsequence sections, we will uses these to identify the slice towers of $H\mM$ and for an algebraically defined class of spectra. We begin with a quite general definition and some elementary properties.

\subsection{Properties of pullbacks}
Let $G$ be an arbitrary finite group, and let $N\lhd G$. Then the functor $X\mapsto X^N$ defines a functor from finite $G$-sets to finite $G/N$-sets, and this functor preserves pullback diagrams. Thus given a Mackey functor $\mM$ on $G/N$, we can compose with the $N$-fixed point functor to get a Mackey functor $\phi_N^\ast\mM$ on $G$. For the orbits, this has a particularly easy description:
\[
\phi_N^\ast\mM(G/H)=\begin{cases}
\mM\big((G/N)/(H/N)\big) & N\subseteq H \\
0 & \text{otherwise.}
\end{cases}
\]
These functors were studied extensively in \cite{GreMay92}.

This algebraic story is part of a much richer narrative in algebraic topology which connects nicely to the geometric fixed points functor and to the slice filtration. Since $N$ is a normal subgroup, the collection 
\[
\mathcal F[N]=\{H\subseteq G | N\not\subseteq H\}
\] 
is closed under subgroups and conjugation. There is, therefore, a universal space $E\mathcal F[N]$ for this family which is built out of cells induced up from members of $\mathcal F[N]$. The above description of the value of the Mackey functor $\phi_N^\ast\mM$ shows that the function spectrum $F(E\mathcal F[N]_+,H\phi_N^\ast\mM)$ is equivariantly contractible, and that we have a natural equivalence
\[
H\phi_{N}^{\ast}\mM\to\tilde{E}\mathcal F[N]\wedge H\phi_{N}^{\ast}\mM.
\]
We use this to motivate a notion of ``pullback'' for a general $G/N$-spectrum $X$.
\begin{defn}
If $X$ is a $G/N$-spectrum, then let $\phi_{N}^{\ast}X$ denote the $G$-spectrum 
\[
\tilde{E}\mathcal F[N]\wedge X,
\] 
where here $X$ is viewed as a $G$ spectrum via the natural quotient map.
\end{defn}
Here we pause. The category of genuine $G$-spectra is a localization of $G$-diagrams in spectra. The localization amounts to ``invert representation spheres'', and this is clearly preserved in passing from $G/N$-diagrams to $G$-diagrams. Smashing with $\tilde{E}\mathcal F[N]$ removes the additional homotopy groups this pullback might produce. Moreover, we have good control over the homotopy groups. Much of this material is also discussed in Lewis-May-Steinberger {\cite[II. \S 9]{LMayS}}, but we briefly sketch the relevant facts.

With this definition, the following is immediate.
\begin{prop}\label{prop:PullBackMackey}
If $\mM$ is a Mackey functor on $G/N$, then 
\[
\phi_{N}^{\ast}H\mM=H\phi_{N}^{\ast}\mM.
\]
\end{prop}

\begin{prop}\label{prop:FixedPoints}
The ``$N$-fixed points'' functor establishes an equivalence between $G/N$-spectra and the image of $\phi_{N}^{\ast}$.
\end{prop}
The image is actually relatively easy to describe: smashing with $\tilde{E}\mathcal F[N]$ is a localization, and the image of $\phi_{N}^{\ast}$ is the subcategory of local objects. This is the heart of Lewis-May-Steinberger's chapter ``the construction of $G/N$-spectra from $G$-spectra'' \cite[II \S9]{LMayS}. The next proposition is quite useful and is a generalization of Greenlees-May's analogous result for pullbacks of {\EM} spectra {\cite[Proposition 10]{GreMay92}}.

\begin{prop}\label{prop:PullBackGeomFP}
For any $X$, we have a natural isomorphism
\[
\big[X,\phi_{N}^{\ast}Y\big]_G \cong \Big[\big(\tilde{E}\mathcal F[N]\wedge X\big)^N,Y\Big]_{G/N}.
\]
\end{prop}
\begin{proof}
Since smashing with $\tilde{E}\mathcal F[N]$ is a localization, we have a natural isomorphism
\[
\big[X,\phi_{N}^{\ast}Y\big]_{G}\cong \big[(\tilde{E}\mathcal F[N]\wedge X),\phi_{N}^{\ast} Y\big]_{G}.
\]
By the equivalence of  homotopy categories described in Proposition~\ref{prop:FixedPoints}, this final group is isomorphic to
\[
\Big[(\tilde{E}\mathcal F[N]\wedge X)^{N},(\phi_{N}^{\ast}Y)^{N}\Big]_{G/N},
\]
since ``$N$-fixed points'' is a left homotopy inverse to $\phi_{N}^{\ast}$, we have the desired result.
\end{proof}
Thus the pullback in equivariant spectra is the right adjoint to the $N$-geometric fixed points functor. This will be essential in our analysis of the role of $\phi_N^{\ast}$ in the slice story.

Perhaps more interesting is a final proposition which links the monoidal structure to pulling back. This follows from basic properties of $\tilde{E}\mathcal F[N]$. This proposition is a kind of Frobenius reciprocity in spectra: the role of restriction is played by $N$-geometric fixed points and the role of transfer is played by pullback.
\begin{prop}\label{prop:PullBackTransfer}
If $X$ is a $G$-spectrum and $Y$ is a $G/N$-spectrum, then we have a natural equivalence
\[
X\wedge\phi_{N}^{\ast}Y\simeq\phi_{N}^{\ast}\big((\tilde{E}\mathcal F[N]\wedge X)^{N}\wedge Y\big).
\]
\end{prop}
\begin{proof}
The space $\tilde{E}\mathcal F[N]$ is a homotopy idempotent under smash product, so this shows us that
\[
X\wedge(\phi_{N}^{\ast}(Y))\simeq (\tilde{E}\mathcal F[N]\wedge X)\wedge(\phi_{N}^{\ast}(Y)).
\]
Now we apply Proposition~\ref{prop:FixedPoints} to deduce an equivalence
\[
\phi_{N}^{\ast}\big((\tilde{E}\mathcal F[N]\wedge X)^{N}\big)\simeq\tilde{E}\mathcal F[N]\wedge X,
\]
since $\tilde{E}\mathcal F[N]\wedge X$ is local. Pulling back is obviously weakly monoidal, so
\[
X\wedge \phi_{N}^{\ast}(Y)\simeq \phi_{N}^{\ast}\big((\tilde{E}\mathcal F[N]\wedge X)^{N}\big)\wedge\phi_{N}^{\ast}(Y)\simeq \phi_{N}^{\ast}\big((\tilde{E}\mathcal F[N]\wedge X)^{N}\wedge Y\big).\qedhere
\]
\end{proof}

\begin{cor}\label{cor:SuspensionsPullback}
If $X$ is a $G$-space, then
\[
X\wedge \phi_{N}^{\ast}Y\simeq \phi_{N}^{\ast}(X^{N}\wedge Y),
\]
so in particular, if $V$ is a representation of $G$, then 
\[
S^{V}\wedge\phi_{N}^{\ast} Y\simeq \phi_{N}^{\ast}\big(S^{V^{N}}\wedge Y\big)
\]
for any $G/N$-spectrum $Y$.
\end{cor}

\begin{cor}\label{cor:PullBackEquivalence}
The inclusion of the zero cell induces an equivalence
\[
H\phi_N^{\ast}\mM\to S^{\rho_{G}-\rho_{G/N}}\wedge H\phi_N^{\ast}\mM
\]
for any $\mM$ on $G/N$.
\end{cor}
\begin{proof}
The $N$-fixed points of $\rho_{G}-\rho_{G/N}$ are $\{0\}$. 
\end{proof}

\begin{cor}\label{cor:PullBackDimension}
The spectrum $H\phi_N^{\ast}\mM$ is in $\tau_{\geq |N|-1}$.
\end{cor}
\begin{proof}
By Theorem~\ref{thm:NegPermutation}, the spectrum $S^{-\rho_{G/N}}$ is in $\tau_{\geq -|G|+|N|-1}$. Thus $S^{\rho_{G}-\rho_{G/N}}$ is in $\tau_{\geq |N|-1}$. Since $H\phi_N^{\ast}\mM$ is $(-1)$-connected, it is in $\tau_{\geq 0}$, and hence the smash product 
\[
S^{\rho_{G}-\rho_{G/N}}\wedge H\phi_N^{\ast}\mM
\]
is in $\tau_{\geq |N|-1}$ by Corollary~\ref{cor:SmashOrderG}. By Corollary~\ref{cor:PullBackEquivalence}, this spectrum is $H\phi_N^{\ast}\mM$.
\end{proof}

\subsection{Pullbacks and Slices}
The relationship between the $N$-geometric fixed points and the pulled-back Mackey functors generates for us a very large collection of slices. Simply put: certain slices for $G/N$ pull back to slices for $G$. We will first show an analogous result for spectra less than or equal to a fixed number.

\begin{thm}\label{thm:PullBackSlices}
If $X$ is $\leq (j-1)$ for $G/N$, then $\phi_N^{\ast}X$ is $\leq (j|N|-1)$ for $G$.
\end{thm}

\begin{proof}
We must show that if $G_{+}\wedge_{H} S^{k\rho_{H}-\epsilon}$ is such that $k|H|-\epsilon \geq j|N|$, then 
\[
[G_{+}\wedge_{H} S^{k\rho_{H}-\epsilon},\phi_{N}^{\ast}X]_{G}=0.
\]

We have a natural equivalence
\[
\big(\tilde{E}\mathcal F[N]\wedge (G_{+}\wedge_{H}S^{k\rho_{H}-\epsilon})\big)^{N}\simeq \begin{cases}
* & N\not\subseteq H, \\
(G/N)_{+}\wedge_{H/N} S^{k\rho_{H/N}-\epsilon} & N\subseteq H.
\end{cases}
\]

We therefore learn that we always have
\[
[G_{+}\wedge_{H} S^{k\rho_{H}-\epsilon},\phi_{N}^{\ast}X]_{G}=\begin{cases}
0 &  N\not\subseteq H, \\
\big[(G/N)_{+}\wedge_{H/N} S^{k\rho_{H/N}-\epsilon}, X\big]_{G/N} & N\subseteq H.
\end{cases}
\]

Now our assumptions about $X$ come into play. Since $N\subseteq H$, if $k|H|-\epsilon\geq j|N|$, then $(k |H/N|-j)|N|\geq \epsilon$. Since the left-hand side is an integer divisible by $|N|$, we conclude that $(k |H/N|-j)|N|\geq |N|\epsilon$, and therefore $k |H/N|-\epsilon\geq j$. By assumption $X$ was $\leq (j-1)$ for $G/N$, and therefore any maps from something greater than or equal to $j$ are zero. Thus if $k|H|-\epsilon\geq j|N|$, then
\[
\big[(G/N)_{+}\wedge_{H/N} S^{k\rho_{H/N}-\epsilon}, X\big]_{G/N}=0,
\]
and we are done.
\end{proof}

\begin{cor}\label{cor:PullBackZeroSlices}
If $H\mM$ is a zero slice for $G/N$, then $H\phi_{N}^{\ast}\mM$ is a $(|N|-1)$-slice for $G$.
\end{cor}
\begin{proof}
Corollary~\ref{cor:PullBackDimension} shows that $H\phi_{N}^{\ast}\mM$ is in $\tau_{\geq (|N|-1)}$. Theorem~\ref{thm:PullBackSlices} shows it is also less than or equal to $(|N|-1)$, and therefore it is a slice.
\end{proof}

This clearly points to a much deeper story.
\begin{conj}\label{conj:PullBackSlices}
The slice tower for a $G/N$-spectrum pulls back to the slice tower for the pullback $G$-spectrum where the $k$-slice becomes the $\big((k+1)|N|-1\big)$-slice.
\end{conj}

We will see in our discussion of geometric spectra that this conjecture holds true for $N=G$. For now, we provide a few more results in this vein. The first is a weaker-than-desired lower bound.

\begin{thm}\label{thm:SliceLowerBound}
If $X$ is in $\tau_{\geq j-1}^{G/N}$, then $\phi_{N}^{\ast}(X)$ is in $\tau_{\geq (j-1)|N|}^{G}$.

If $X$ is actually in the localizing subcategory generated by $G_+\wedge_H S^{k\rho_H}$ for $k\cdot |H|\geq j-1$, then $\phi_N^\ast(X)$ is in $\tau_{\geq j|N|-1}$.
\end{thm}
\begin{proof}
Pulling back preserves cofiber sequences and extensions, so we need only show that this is true for the generators of $\tau_{\geq j-1}^{G/N}$. Since for any $H$ containing $N$, we have
\[
\phi_{N}^{\ast}\big((G/N)_{+}\wedge_{H/N} X\big)\simeq G_{+}\wedge_{H}\phi_{N}^{\ast}X
\]
(the only possibly non-obvious step here is that $\tilde{E}\mathcal F[N]$ for $G$ restricts to $\tilde{E}\mathcal F[N]$ for $H$), it suffices to prove this for $H=G$. We therefore show that for any $k$ and $\epsilon=0,1$ such that $k|G/N|-\epsilon \geq j-1$, we have $\phi_{N}^{\ast}(S^{k\rho_{G/N}-\epsilon})$ is in $\tau_{\geq (j-1)|N|}$. By Corollary~\ref{cor:SuspensionsPullback}, we have an equivalence
\[
\phi_{N}^{\ast}(S^{k\rho_{G/N}-\epsilon})\simeq S^{k\rho_{G}-\epsilon}\phi_{N}^{\ast}(S^{0}).
\]
Since $\phi_{N}^{\ast}(S^{0})$ is $(-1)$-connected, it is in $\tau_{\geq 0}$. Thus we need only show that 
\[k|G|-\epsilon \geq (j-1)|N|.\] 
This, however, is obvious.

For the second part, we observe that $\epsilon$ is always zero here. Now we are looking at a $k\rho_G$-suspension of $\phi_N^\ast(S^0)$. Since
\[
\phi_N^\ast(S^0)\simeq S^{\rho_G-\rho_{G/N}}\phi_N^\ast(S^0),
\]
we conclude that $\phi_N^\ast(S^0)$ is in $\tau_{\geq |N|-1}$. Hence
\[
\sdim\big(S^{k\rho_G}\phi_N^\ast(S^0)\big)=k\cdot |G|+\sdim\big(\phi_N^\ast(S^0)\big)\geq j|N|-1.\qedhere
\]
\end{proof}

\begin{remark}
The second half of the previous theorem shows that if we consider the localizing categories generated only by the regular representation spheres, then Conjecture~\ref{conj:PullBackSlices} is true. Ullman has recently studied this ``regular slice filtration'', and he independently proved that in this context, slices for $G/N$ pull back to slice for $G$ \cite{Ul12}.
\end{remark}

Suspension invariance and the way suspensions pass through pullbacks show that in two cases, we completely understand how slices pullback.
\begin{prop}
If $(j-1)\equiv 0,-1\mod |G/N|$, then $(j-1)$-slices for $G/N$ pullback to $(j|N|-1)$-slices for $G$.
\end{prop}
\begin{proof}
If $(j-1)=k|G/N|-\epsilon$, then a $(j-1)$-slice is a $(k\rho_{G/N})$-suspension of a $(-\epsilon)$-slice. By Corollary~\ref{cor:SuspensionsPullback}, $S^{k\rho_{G/N}}$ pulls back to $S^{k\rho_{G}}$. Corollary~\ref{cor:PullBackZeroSlices} shows that zero slices pullback to $(|N|-1)$-slices. Proposition~\ref{prop:PullBackMackey} shows that $(-1)$-slices pullback to $(-1)$-slices, and the result is proved.
\end{proof}

\begin{cor}
If $[G:N]=2$, then slices for $G/N$ pull back to slices for $G$.
\end{cor}

\section{The slices of $H\mM$ for cyclic $p$-groups}\label{sec:Burnside}
We will use our understanding of pullbacks to determine all of the slices of $H\mA$ and $H\mM$ for an arbitrary cyclic $p$-group, where $\mA$ is the Burnside Mackey functor and where $\mM$ is arbitrary. Our method uses some amusing properties of $\mA$ which are specific to cyclic $p$-groups, namely the simplicity of the augmentation ideal. For a general $\mM$, the story is slightly trickier, and interestingly, our slice filtration on $H\mM$ is essentially the same as the filtration used by Barwick in his proof of the Carlsson conjecture \cite{Barwick}.
\subsection{The slice tower for $H\mA$}
For $G=C_{p^n}$, the slices of the {\EM} spectrum for the Burnside Mackey $\mA$ functor are actually shockingly simple and the slice filtration arises in an obvious algebraic way. 

Since $H\mA$ is $(-1)$-connected, we know that $H\mA$ is in $\tau_{\geq 0}$. We also know that the zero slice assigns to $G/H$ the group
\[
Im\big(\mA(G/H)\to\mA(G/\{e\})\big)=Im\big(A(H)\to \Z\big),
\]
where the map is the augmentation sending a finite $H$-set to its cardinality. This is the well-known constant Mackey functor $\mZ$: the value on $G/H$ is $\mathbb Z$, all restriction maps are the identity, and transfers are multiplication by the index. Let $\mI$ denote the augmentation ideal, the kernel of the map $\mA \to \mZ$. This we can write as a sum of pullbacks for $G=C_{p^{n}}$.

For $G=C_{p^n}$, we have a single subgroup $C_{p^k}$ for each $0\leq k\leq n$. Instead of the somewhat cumbersome $\phi_{C_{p^k}}^\ast$, we will write $\phi_k^\ast$.

\begin{prop}
If $\mI$ is the augmentation ideal of the Burnside Mackey functor for $C_{p^n}$, then we have an isomorphism of Mackey functors
\[
\mI=\bigoplus_{k=1}^{n}\phi_k^\ast \mZ^\ast,
\]
where $\mZ^\ast$ is the dual to the constant Mackey functor $\mZ$ on $C_{p^{n-k}}$.
\end{prop}
\begin{proof}
Consider the subMackey functor generated by the element $[C_{p^{k+1}}/C_{p^k}]-p$ in $\mI(C_{p^n}/C_{p^{k+1}})$. The transfer of this element to $C_{p^m}$ is $[C_{p^m}/C_{p^k}]-p[C_{p^m}/C_{p^{k+1}}]$, and the collection of all elements generated by these elements for fixed $k$ is clearly a subMackey functor. This restricts to zero in $\mI(C_{p^n}/C_{p^k})$, and it is immediate that this subMackey functor is $\phi_k^\ast\mZ^\ast$. Moreover, for fixed $m$ and varying $k$, the elements $[C_{p^m}/C_{p^k}]-p[C_{p^m}/C_{p^{k+1}}]$ are linearly independent and generate $\mI(C_{p^n}/C_{p^m})$.
\end{proof}

\begin{cor}
We have an equivalence
\[
H\mI=\bigvee_{k=1}^n H\phi_k^\ast\mZ^\ast.
\]
\end{cor}

We know also that $H\mI$ is in $\tau_{\geq 1}$, since it is the fiber of the map from $H\mA$ to its zero slice $H\mZ$. By Corollary~\ref{cor:PullBackZeroSlices}, $H\mI$ is a wedge of slices. These are therefore the slices of $H\mA$, and the filtration by slices of degree less than or equal to $m$ is the $m$-slice section.

What is perhaps most curious here is that all of the slice sections are Mackey functors associated to a natural filtration of the Burnside ring by summands. By considering degrees, we see the following.

\begin{thm}
If $n$ is not of the form $p^{k}-1$, then the map $P^{n}(H\mA)\to P^{n-1}(H\mA)$ is an equivalence. We also have an equivalence
\[
P^{p^{k}-1} H\mA=H\left(\mA\middle/\bigoplus_{j=k+1}^{n}\phi_k^{\ast}\mZ^{\ast}\right).
\]
\end{thm}

In particular, the slice spectral sequence for $H\mA$ is very simple. All groups are concentrated in dimension zero, and their filtrations are exponentially distributed. In fact, the groups lie at the very edge of the vanishing regions for various subgroups, which provides an interpretation for the vanishing lines in the slice spectral sequence as lines detecting when restrictions vanish.

Since every Mackey functor is a module over the Burnside ring, the naturality of the slice tower tells us that the corresponding filtration of an arbitrary Mackey functor $\mM$ is a good first approximation to the slice tower for $H\mM$. Each of the stages of the slice tower for $H\mA$ is the {\EM} spectrum for a Green functor (in fact, a Tambara functor), so the $\mA$-module structure of a Mackey functor $\mM$ shows that the slices of $H\mM$ are in the obvious module categories. We can do better, though, and explicitly determine all of the slices of $H\mM$ for a cyclic $p$-group.

\subsection{The slice of $H\mM$ for a cyclic $p$-group}
We begin by defining a filtration on a Mackey functor $\mM$. For obvious reasons, we will call it the coslice filtration.

\begin{defn}
If $\mM$ is a Mackey functor, we define a filtration of $\mM$ by saying that the $k$\textsuperscript{th} filtered piece $F^{k}\mM$ is the subMackey functor of $\mM$ generated by all elements which restrict to zero in all subgroups of order at most $k$.

We also let $F^{0}\mM$ be $\mM$.
\end{defn}

\begin{remark}
As stated, this definition obviously works for a general finite group $G$. We return to the question of how much of this section holds in that case in the last section.
\end{remark}

For a cyclic $p$-group, it is obvious that $F^{k}\mM=F^{k-1}\mM$ unless $k$ is a power of $p$. Similarly, by construction, $F^{p^{k}}\mM$ is pulled-back from $C_{p^{n}}/C_{p^{k}}$. We can do significantly better, though.

\begin{prop}\label{prop:FiltrationQuotients}
If $k\geq 1$, then the spectrum $H\big(F^{p^{k}-1}\mM/F^{p^{k}}\mM\big)$ is the pullback of a zero-slice for $C_{p^{n}}/C_{p^{k-1}}$.

If $k=0$, then the spectrum $H\big(\mM/F^{1}\mM\big)$ is a zero-slice.
\end{prop}
\begin{proof}
We prove the second part first, as the first follows easily from this.

By Theorem~\ref{thm:zeroslices}, zero slices are distinguished by the fact that all restriction maps to the trivial group are injections. The subMackey functor $F^{1}\mM$ is the subMackey functor of all elements which restrict to zero in the trivial group, and so the quotient $\mM/F^{1}\mM$ is a zero slice.

For the first part, we note that $F^{p^{k}-1}\mM$ is pulled-back from $C_{p^{n}}/C_{p^{k}}$. In particular, it is the pullback of $F^{0}\mN$ for some $\mN$ on $C_{p^{n}}/C_{p^{k}}$ Similarly, $F^{p^{k}}\mM$ is the pullback of $F^{1}\mN$. The quotient is therefore the pullback of $\mN/F^{1}\mN$, which is the pullback of a zero slice.
\end{proof}

\begin{cor}\label{cor:EMLayers}
For any $\mM$ over $C_{p^{n}}$, $H\big(F^{p^{k}-1}\mM/F^{p^{k}}\mM\big)$ is a $(p^{k}-1)$-slice.
\end{cor}

\begin{thm}
For $G=C_{p^{n}}$, then $r$\textsuperscript{th} slice of $H\mM$ is
\[
P_{r}^{r}(H\mM)=H\big(F^{r}\mM/F^{r+1}\mM\big),
\]
and the slice sections are
\[
P^{r}(H\mM)=H\big(\mM/F^{r+1}\mM\big).
\]
\end{thm}
\begin{proof}
We rewrite the coslice filtration of $\mM$ as a tower with limit $\mM$: $P^{r}\mM=\mM/F^{r+1}\mM$. We have obvious short exact sequences 
\[
F^{r}\mM/F^{r+1}\mM\to \mM/F^{r+1}\mM\to \mM/F^{r}\mM.
\]
Applying $H(-)$ to this tower of Mackey functors produces a tower of fibrations in spectra. Since our group is a cyclic $p$-group, we know that these only change when $(r+1)$ is a power of $p$. By Corollary~\ref{cor:EMLayers}, the fibers in the tower are all slices, and by construction, they are arranged in increasing order. We therefore conclude that this tower of {\EM} spectra is the slice tower.
\end{proof}

Now we give a small warning (which is in some sense obvious). While it is the case that the Mackey functors $\mM/F^{r}\mM$ are modules over $\mA/F^{r}\mA$ (and this is true for any finite group, in fact), in general
\[
\mM/F^{r}\mM\not\cong \mM\boxprod \mA/F^{r}\mA.
\]
There is always a map from the right to the left, but even for zero slices, it can fail to be an isomorphism.

\subsection{An example}
To facilitate understanding of the coslice filtration on a Mackey functor, we get our hands dirty with an example. Let $G=C_{2}$, and we use the standard notation for a Mackey functor:
\[
\xymatrix@!R=1pt{
{} & {\mM(G/G)}\ar@(l,l)[dd]_{r} \\ 
{\mM=} & {} \\
{} & {\mM(G/\{e\})} \ar@(r,r)[uu]_{t} \ar@(dl,dr)[]_{\gamma}}
\]
where $r$ is the restriction, $t$ is the transfer, and $\gamma$ generates the Weyl group.

Our example will be the Mackey functor $\mE$ defined by
\[
\xymatrix@!R=1pt{
{} & {} & {\Z\oplus\Z/2}\ar@(l,l)[dd]_{\begin{bmatrix} 1 & 0 \end{bmatrix}} \\
{\mE=} & {} & {} \\
{} & {} & {\Z/2}\ar@(r,r)[uu]_{\begin{bmatrix} 0 \\ 1 \end{bmatrix}} \ar@(dl,dr)[]_{1}}
\]

By definition, $F^{0}\mE=\mE$. Similarly, $F^{1}\mE$ is the subMackey functor generated by all elements which restrict to zero in $\mE(G/\{e\})$. This gives the following Mackey functors for $F^{1}\mE$ and the quotient $\mE/F^{1}\mE$:
\[
\xymatrix@!R=1pt{
{} & {2\Z\oplus \Z/2} \ar@(l,l)[dd]_{0} \\
{F^{1}\mE=} & {} \\
{} & {0} \ar@(r,r)[uu]_{0} \ar@(dl,dr)[]_{0}
} 
\quad
\xymatrix@!R=1pt{
{} & {\Z/2}\ar@(l,l)[dd]_{1} \\
{\mE/F^{1}\mE=} & {} \\
{} & {\Z/2}\ar@(r,r)[uu]_{0} \ar@(dl,dr)[]_{1}
}.
\]
The slice tower is in this case simply the fiber sequence $HF^{1}\mE\to H\mE\to H(\mE/F^{1}\mE)$.

The module $\mE$ also fails to satisfy $\mE/F^{1}\mE\cong\mE\boxprod \mZ$. In this case, the latter is
\[
\xymatrix@!R=1pt{
{} & {\Z/4}\ar@(l,l)[dd]_{1} \\
{\mE\boxprod\mZ=} & {} \\
{} & {\Z/2}\ar@(r,r)[uu]_{2} \ar@(dl,dr)[]_{1}
}
\]
and the map to the zero slice of $\mE$ is the obvious quotient map.

\section{Geometric spectra}\label{sec:Geometric}

The notion of pullback has already proved useful. We focus now on a special case, ``geometric spectra'', for which the slice tower is a reindexed form of the Postnikov tower. These are actually distinguished by an algebraic condition on their Mackey functor homotopy groups.

\begin{defn}
A Mackey functor $\mM$ is concentrated on $G/G$ if
\[
\mM(G/H)=0
\]
for all proper subgroups $H$.
\end{defn}
Thus $\mM$ is concentrated on $G/G$ if $\mM=\phi_G^\ast M$ for some $M$ an abelian group (which is a Mackey functor for the trivial group). Thus much of this section simply exploits the fact that Conjecture~\ref{conj:PullBackSlices} is true in this context, and we work out many of the implications of this.

We have already encountered a number of examples.
\begin{example}\mbox{}
\begin{enumerate}
\item If $G=C_p$, then the augmentation ideal $\mI$ is concentrated on $G/G$.
\item For any $G$, $\underline{\pi_0}F(S^{\rho_{G}-1},H\mM)=\underline{H}^0(S^{\rho_G-1};\mM)$ is concentrated on $G/G$ for any Mackey functor $\mM$.
\end{enumerate}
\end{example}

These examples give some of the general algebraic flavor.
\begin{prop}\label{prop:EPmM}
Every Mackey functor $\mM$ has a largest canonical quotient $\EP\otimes\mM$ such that $\EP\otimes\mM$ is geometric. If $\mM$ is a Green functor, then so is $\EP\otimes\mM$.
\end{prop}
\begin{proof}
To form $\EP\otimes\mM$, we quotient $\mM$ by the sub-Mackey functor generated by $\mM(G/H)$ for all proper subgroups $H$. If $\mM$ is a Green functor, then by definition, this is a Mackey ideal, and hence $\EP\otimes\mM$ is a Green functor. 
\end{proof}

\begin{remark}
If $\mM(G/G)$ is a division ring and $\mM$ is geometric, then $\mM$ is a Mackey division ring. These have been studied extensively by Lewis and Oruc \cite{Lewis:Green}, \cite{Oruc}. We believe that pull-back and the fixed point Mackey functors associated to Galois field extensions exhaust all Mackey fields.
\end{remark}

We have chosen this notation to underscore the close connection between Mackey functors concentrated on $G/G$ and geometric fixed points. This connection will be fleshed out further shortly.

\begin{remark}
The previous proposition is really a statement about pulled-back Mackey functors. There is an identical proof which shows that every $\mM$ has a canonical largest quotient $\tilde{E}\mathcal F[N]\otimes \mM$ which is pulled back from $G/N$. Additionally, this is a Green functor if $\mM$ is. This is also the localization onto the full-subcategory of Mackey functors for which the pullback is the left adjoint, as described in \cite[Proposition 4]{GreMay92}.
\end{remark}

The second example above shows us the following proposition.
\begin{prop}
Every $\mM$ has a canonical maximal subMackey functor $P\mM$ such that $P\mM$ is concentrated on $G/G$.
\end{prop}
This reinforces the close connection between Mackey functors concentrated on $G/G$ and the slice filtration (since $P\mM$ is the top slice of $H\mM$, $P_{|G|-1}^{|G|-1}H\mM$).

We now turn more fully to topological considerations, directly linking this algebraic discussion to slice and geometric fixed point constructions. We have two essentially equivalent consequences of the definition: if $\mM$ is concentrated on $G/G$, then
\begin{enumerate}
\item $G/H_{+}\wedge H\mM\simeq\ast$ and
\item $F(G/H_+,H\mM)\simeq \ast$
\end{enumerate}
for all proper subgroups $H$. 

Since the cofiber of the inclusion $S^{V^{G}}\hookrightarrow S^{V}$ is built out of cells with proper stabilizer subgroup, the previous observation shows that
\[
S^{V^{G}}\wedge H\mM\simeq S^{V}\wedge H\mM,
\]
and the equivalence is induced by the inclusion of the fixed point sphere. Coupled with the stability of slices under suspension by copies of the regular representation, this gives a huge list of slices.

\begin{thm}\label{thm:Slices}
If $\mM$ is concentrated on $G/G$, then $\Sigma^{n-1}H\mM$ is a $(n|G|-1)$-slice.
\end{thm}
\begin{proof}
For any Mackey functor $\mN$, $\Sigma^{-1}H\mN$ is a $(-1)$-slice. We therefore learn that
$\Sigma^{n\rho_{G}-1}H\mN$ is a $(n|G|-1)$-slice, and if $\mM$ is concentrated on $G/G$, then 
\[
\Sigma^{n\rho_{G}-1}H\mM\simeq\Sigma^{n-1}H\mM.\qedhere
\]
\end{proof}

\begin{remark}
It is not the case that a Mackey functor concentrated on a proper subgroup has this property. For $G=C_{2}$, the fixed point Mackey functor associated to the sign representation is concentrated on $G/\{e\}$, but it is easy to produce suspensions which are not slices.
\end{remark}

\begin{remark}
The analogous result for pullbacks is that if $\mM$ is a zero slice for $G/N$, then $\Sigma^{k\rho_{G/N}}H\phi_{N}^{\ast}\mM$ is an $(k|G|+|N|-1)$-slice. Similarly, for any Mackey functor $\mM$, $\Sigma^{k\rho_{G/N}-1}H\phi_{N}^{\ast}\mM$ is a $(k|G|-1)$-slice. The proof is immediate from Corollary~\ref{cor:SuspensionsPullback}.
\end{remark}

This property that the function spectrum from $G/H_{+}$ to $H\mM$ is contractible for any $\mM$ concentrated on $G/G$ is shared by $\EP$, where $\cP$ is the family of proper subgroups. As this spectrum is that with which we smash to get the geometric fixed points, we give a topological version of our earlier algebraic statement.

\begin{defn}
A $G$-spectrum $X$ is geometric if $\underline{\pi_\ast}(X)$ is concentrated on $G/G$.
\end{defn}

Smashing with $\EP$ effects the localization nullifying all maps from induced objects, so we derive the following equivalent form immediately.

\begin{prop}
A spectrum is geometric if and only if the natural map
\[
X\to\EP\wedge X
\]
is a $G$ weak equivalence.
\end{prop}
This gives the reason for calling such spectra ``geometric'': the fixed and geometric fixed points agree. Since $\EP$ is a smash idempotent in $G$-spectra, we have a lot of geometric spectra.
\begin{cor}
For all $Y$, $\EP\wedge Y$ is geometric.
\end{cor}
In fact, all geometric spectra essentially arise in this way (due to the equivalence described in Proposition~\ref{prop:FixedPoints}). If $\mM$ is a Mackey functor, then the Mackey functor $\EP\otimes\mM$ of Proposition~\ref{prop:EPmM} is $\underline{\pi_0}(\EP\wedge H\mM)$. Moreover, for $X$ a $(-1)$-connected $G$-spectrum, 
\[
\underline{\pi_{0}(\EP\wedge X)}=\EP\otimes \underline{\pi_{0}X}.
\]

While all geometric spectra are of this form, they arise from seemingly unrelated ways. The following proposition is obvious, but we have found it sufficiently useful that it bears repeating.

\begin{prop}
If $f\colon X\to Y$ is an equivariant map such that
\[
\underline{\pi_\ast}(X)(G/H)\to\underline{\pi_\ast}(Y)(G/H)
\]
is an isomorphism for all proper subgroups $H$, then the fiber of $f$ is geometric.
\end{prop}

This seemingly simple condition of ``geometric'' actually allows us to completely deduce the slices of a geometric spectrum $X$.

\begin{thm}\label{thm:GeomSliceTower}
If $X$ is a geometric spectrum, then the slice tower of $X$ is a reindexed form of the Postnikov tower of $X$: the $(n-1)$\textsuperscript{st} Postnikov layer is the $(n|G|-1)$\textsuperscript{st} slice  and all other slices are contractible.
\end{thm}
\begin{proof}
We first note that the Postnikov layers are in fact slices by Theorem~\ref{thm:Slices}. The slice tower is a nullification tower, nullifying induced up representation spheres. Since maps from anything induced from a proper subgroup are automatically null (this is a restatement of $X$ being geometric), we need only consider the effect of nullifying maps from regular representation spheres. In particular, we see that the slices only change in dimensions congruent to $-1$ or $0$ modulo the order of $G$.

Now we apply the same argument used for Theorem~\ref{thm:Slices}: the inclusion $S^{n-1}$ into $S^{n\rho_{G}-1}$ is an inclusion modulo induced cells, and therefore the space of maps from $S^{n-1}$ into $X$ is the same as the space of maps of $S^{n\rho_{G}-1}$ into $X$. The nullification towers are therefore the same, giving the result.
\end{proof}

\begin{remark}
In \cite[Proposition 4.41]{HHR:Kervaire}, it is shown that if the layers of a tower are slices of increasing slice-connectivity, then the tower is the slice tower. Theorem~\ref{thm:Slices} shows that the layers of the Postnikov tower are increasing slices, and so we conclude that the two towers are the same.
\end{remark}

This theorem gives us a way to interpret the slice tower for a general spectrum $Y$. The Mackey functor homotopy groups of $Y$ are essentially determined by the homotopy groups of the $H$-geometric fixed points of $Y$ for all subgroups $H$. Each of these are restrictions of $Y$ to subgroups, followed by smashing with the appropriate $\EP$, and since smashing with $\EP$ produces a geometric spectrum, the slices of these are a smeared version of the ordinary Postnikov sections. Thus we learn that the slice tower is essentially an aggregation of stretched out Postnikov sections, scaled by the order of the appropriate subgroups.

\section{Here be dragons}

We finish with a few observations and conjectures. These fall into two flavors: algebraic and topological. As is common in algebraic topology, both are underlain by the topology of the slice filtration.

On the algebraic side, we have the general question of the algebraic analogue of the slice filtration. We defined a decreasing filtration of a general Mackey functor, and for cyclic $p$-groups, this gives the slice filtration. The definition, however, was independent of the structure of the finite group: the $k$\textsuperscript{th} filtered piece is the subMackey functor generated by all elements which restrict to zero in all subgroups of order at most $k$. The obvious question is if the {\EM} spectra associated to the filtration quotients realize the slices of $H\mM$. Our analysis of slices breaks down here, since we do not understand how to handle non-normal subgroups and pullbacks therefrom\footnote{Since the submission of this paper, Ullman has used the regular slice filtration to prove this conjecture and the following \cite{Ul12}.}.

There is an additional subtlety: it is not immediately clear that the slices are all {\EM} spectra. For $H\mA$ and $H\mM$ for $G=C_{p^{n}}$, we showed this directly. For other spectra, this is less clear (though we know that {\em{eventually}} the layers of the slice tower will be {\EM} spectra). The process of forming the intermediate stages could introduce and then kill higher homotopy groups. While we believe all slices to be {\EM} spectra, we do not have an easy argument showing it.

On the topological side, we have the general question of the relationship between the slice tower of a quotient group $G/N$ and the slice tower for $G$. We showed in Theorem~\ref{thm:PullBackSlices} that if $X$ is less than or equal to $(j-1)$ for $G/N$ then $\phi_{N}^{\ast}(X)$ is less than or equal to $j|N|-1$ for $G$. This provides the upper bound. We do not  have the corresponding precise lower bound except in special cases. In other words, if $X$ is a $(j-1)$-slice, then we do not know that $\phi_{N}^{\ast}(X)$ is in $\tau_{\geq (j|N|-1)}$, which would imply that $\phi_{N}^{\ast}(X)$ is a $(j|N|-1)$-slice. The only examples in which we could completely understand the connection are the geometric spectra, and here we see that this exactly holds.

The notion of ``pullback'' used here is not the naive one. We have chosen to couple the pulling-back with the $N$-geometric fixed points to make more conceptual the connections between the homotopy groups and the obvious effect on Mackey functors. This has the added advantage of establishing the $N$-geometric fixed points as a (homotopy) left adjoint to the $N$ pullback functor. Coupled with our earlier conjecture (Conjecture~\ref{conj:PullBackSlices}), this suggests a very clean and detailed story linking the slice tower to various flavors of geometric fixed points and vastly generalizing Theorem~\ref{thm:GeomSliceTower} and the remarks thereafter.

\bibliographystyle{plain}

\bibliography{math}

\end{document}